\documentclass[12pt]{article}
\usepackage[centertags]{amsmath}
\usepackage{amsfonts}
\usepackage{amssymb}
\usepackage{amsthm}
\usepackage{mathrsfs}
\usepackage{graphicx}
\usepackage{pgf,tikz}
\usetikzlibrary{arrows}

\theoremstyle{plain}
\newtheorem{thm}{Theorem}[section]
\newtheorem{cor}[thm]{Corollary}

\newtheorem{prop}[thm]{Proposition}

\theoremstyle{definition}

\newtheorem{exam}[thm]{Example}

\theoremstyle{remark}
\numberwithin{equation}{section}

\newcommand{\beast}{\begin{eqnarray*}}
\newcommand{\eeast}{\end{eqnarray*}}

\title{Strongly Regular Graphs \\ as Laplacian Extremal Graphs}

\author{Fan-Hsuan Lin\footnote{Corresponding author. E-mail address: fanhsuan.am03g@nctu.edu.tw (F.-H Lin).}
~\footnote{Department of Applied Mathematics, National Chiao Tung University, Taiwan R.O.C..}
\and Chih-wen Weng$^\dag$}

\begin{document}
\maketitle

\bibliographystyle{plain}

\begin{abstract}
The Laplacian spread of a graph is the difference between the largest eigenvalue and the second-smallest eigenvalue of the Laplacian matrix of the graph.
We find that the class of strongly regular graphs attains the maximum of largest eigenvalues, the minimum of second-smallest eigenvalues of Laplacian matrices and hence the maximum of Laplacian spreads among all simple connected graphs of fixed order, minimum degree, maximum degree, minimum size of common neighbors of two adjacent vertices and minimum size of common neighbors of two nonadjacent vertices. Some other extremal graphs are also provided.
\end{abstract}

{\bf keyword}
Laplacian matrix, Laplacian spread, strongly regular graph.

\section{Introduction}\label{s1}
Let $G=(V,E)$ be a simple connected  graph of order $n$  with vertex set $V=\left\{1,2,\cdots,n\right\}$ and edge set $E$. Let $A=A(G)$ be the {\it adjacency matrix} of $G$, i.e. the binary matrix with $ij$-entry  $1$ iff   $i$ and $j$ are distinct and adjacent. The {\it degree} $d_i$ of vertex $i \in V$ is the number $|G_1(i)|$, where $G_1(i)$ is the set of vertices which are adjacent to $i$. Let $D(G)$=diag($d_1,d_2,\cdots,d_n$) be the diagonal matrix with entries $d_1,d_2,\cdots,d_n$ in the diagonal. Then the matrix   $$ L(G)=D(G)-A(G)  $$  is called the {\it Laplacian matrix} of $G$. We call the eigenvalues of $L(G)$ the {\it Laplacian eigenvalues} of $G$. It is well-known that $L(G)$ is symmetric, positive semidefinite, and every row-sum being zero \cite{1973}, so we denote the Laplacian eigenvalues of $G$ in nonincreasing order as $\ell_1(G) \geq \ell_2(G) \geq \cdots \geq \ell_n(G)=0.$ The eigenvalues $\ell_1(G),$ $\ell_{n-1}(G)$ and $\ell_n(G)$ are called {\it Laplacian index}, {\it algebraic connectivity} and {\it trivial eigenvalue},  respectively.  The {\it Laplacian spread} of  $G$ is defined as $\mathscr{S}_L(G):=\ell_1(G)-\ell_{n-1}(G).$

A graph $G$ is called {\it $k$-regular} if any vertex has degree $k$. Moreover $G$ is called
{\it strongly regular} with parameters $(n,k,\lambda,\mu)$ if $G$ is a $k$-regular graph with order $n$ which has $\lambda$ (resp. $\mu$) common neighbors of any pair of two adjacent (resp. nonadjacent) vertices.
The Laplacian matrix of a $k$-regular graph is $kI-A$, so its eigenvalues are easily obtained from those of adjacency matrix. It is well-known that a strongly regular graph with parameters $(n,k,\lambda,\mu)$ and $n\not=k+1$  has two distinct nontrivial Laplacian eigenvalues $\ell_1, \ell_{n-1},$ and indeed
\begin{equation}\label{srg} \ell_1(G), \ell_{n-1}(G)=\dfrac{2k-\lambda+\mu\pm\sqrt{(\lambda-\mu)^2+4(k-\mu) }}{2}.\end{equation}
See for example \cite[Chapter 21]{LW:2001}. Note that $G$ is strongly regular iff its complement $G^c$ is strongly regular \cite[Theorem~1.3.1]{bcn}.

Let $\delta:=\min_{i\in V} d_i$ and $\Delta:= \max_{i\in V} d_i.$ Motivated by the definition of strongly regular graphs, we define another two graph parameters $\lambda(G)$ and $\mu(G)$ for any graph $G:$
\begin{align*}
\lambda(G):=\min_{ij\in E} |G_1(i)\cap G_1(j)|,\qquad
\mu(G):= \min_{ij\not\in E} |G_1(i)\cap G_1(j)|.
\end{align*}

In Section~\ref{s2}, we review some previously known bounds for the Laplacian index $\ell_1(G)$ and algebraic connectivity $\ell_{n-1}(G).$ Our main theorem, Theorem~\ref{mainthm} in Section~\ref{s3}, is an inequality involving the graph parameters $n$, $d_i$, $m_i$, $\lambda(G)$, $\mu (G)$ and a nontrivial eigenvalue $\ell$ with its corresponding eigenvector $(x_1, x_2, \ldots, x_n)^T$ of a graph $G,$ where $m_i=(\sum_{ji\in E}d_j)/d_i$ is called the {\it average $2$-degree} of a vertex $i$. We give two ways to eliminate the eigenvector parameters in the above inequality, and provide many lower bounds and upper bounds of the Laplacian eigenvalue $\ell$. Among them upper bounds for $\ell=\ell_1(G)$ and lower bounds  of $\ell=\ell_{n-1}(G)$ are of most interest. See Corollary~\ref{eig}, and \eqref{ee3}, \eqref{ee5} in Corollary~\ref{Ddelta}.  We find that the class of strongly regular graphs attains both bounds. In Section~\ref{s4}, we provide more extremal graphs, attaining bounds  of some but not all of the inequalities in Section~\ref{s3}. In Section \ref{s5}, 
the Laplacian index $\ell_1(G)$ and algebraic connectivity $\ell_{n-1}(G)$ for all connected $(n-3)$-regular graphs $G$ of order $n$ are determined. With a class of exceptions, they are extremal for all the inequalities obtained in this paper.

\section{Some known bounds}\label{s2}

We recall some basic properties of Laplacian matrices  \cite{newman} and known bounds about Laplacian index $\ell_1(G)$ and algebraic connectivity $\ell_{n-1}(G)$ in this section. For the basic properties, one can find from \cite{newman}.

The most important important property of $L(G)$ probably is
\begin{equation}\label{imp}X^\top L(G)X=\sum\limits_{j<k, jk\in E}(x_j-x_k)^2,\end{equation}
where $X$ is a column vector.
As the property $L(G)+L(G^c)=nI-J$, where $J$ is the all one matrix, the Laplacian matrices $L(G)$ and $L(G^c)$ of $G$ and its complement $G^c$ respectively share the same set of eigenvectors. Hence if $\ell$ is a nontrivial eigenvalue of $L(G)$ with associated eigenvector $X,$ then $n-\ell$ is an eigenvalue of $L(G^c)$ with the same associated eigenvector $X$. Thus an upper bound of $\ell_1(G^c)$ also gives a lower bound of $\ell_{n-1}(G).$

In 1973\cite{1973}, Fiedler showed the following upper bounded about $\ell_{n-1}(G)$
\begin{equation}\label{delta}
\ell_{n-1}(G)\leq \kappa(G) \leq \delta,
\end{equation}
intriguing the study of algebraic connectivity,
where $\kappa(G)$ is the vertex connectivity of $G$. 
Note that $\ell_{n-1}(G)=0$ iff $G$ is disconnected.

The upper bounds of Laplacian index $\ell_1(G)$ were studied by many authors.
In 1985\cite{1985}, Anderson and Morley showed that
\begin{equation}\label{1985}
\ell_1(G) \leq \underset{ij\in E }{\max}\left\{ d_i+d_j \right\}.
\end{equation}
Note that
   $d_i+m_i=d_i+(\sum_{ji\in E}d_j)/d_i\leq d_i + \underset{ji\in E}{\max}\{d_j\}\leq \underset{j i\in E}{\max}\{d_i+d_j\}.$
In 1998\cite{1998}, Merris improved the bound in(\ref{1985}) by showing
\begin{equation}
\ell_1(G) \leq \underset{i \in V}{\max}\left\{ d_i+m_i \right\}.
\end{equation}
As another way to improve the bound in(\ref{1985}), in 2000\cite{2000}, Rojo et al. showed
\begin{equation} \label{2000}
\ell_1(G) \leq \underset{ij\in E}{\max}\left\{ d_i+d_j-|G_1(i) \cap G_1(j)| \right\}.
\end{equation}
In 2001\cite{2001}, Li and Pan gave a bound, as follows
\begin{equation} \label{2001}
\ell_1(G) \leq \underset{i \in V}{\max}\left\{ \sqrt{2d_i(d_i+m_i)} \right\}.
\end{equation}
In 2004\cite{2004}, Zhang showed the following result, which is always better than the bound (\ref{2001}).
\begin{equation}\label{2004}
\ell_1(G) \leq \underset{i \in V}{\max}\left\{ d_i+\sqrt{d_i m_i}\right\}.
\end{equation}
One of our results in Corollary~\ref{eig} is an extension of (\ref{2004}).

For the lower bound of $\ell_1(G)$ in 1994\cite{1994}, Grone and Merris showed that
\begin{equation}\label{Delta}
\ell_1(G)\geq \Delta+1.
\end{equation}

The studies of Laplacian spread $\mathscr{S}_L(G)=\ell_1(G)-\ell_{n-1}(G)$ can be found in \cite{btf:2009, cw:2009, flt:2010, fxwl:2008, yl:2010}. They are interested in the graphs with a few more edges than the number of edges in  a tree. The results in this paper with different favor are about graphs of higher vertex connectivity.

\section{New bounds}\label{s3}

The following is our main theorem.

\begin{thm}\label{mainthm}
Let $G=(V,E)$ be a simple connected graph of order $n$. Let $\ell$ be a nontrivial Laplacian eigenvalue of $G$ with associated eigenvector $X=(x_1,x_2,\cdots,x_n)^\top$. Let $d_i$ and $m_i$ bee degree and average $2$-degree respectively of vertex $i\in V$, and let $\lambda\leq\lambda(G)$ and  $\mu\leq\mu(G)$ be two given numbers. Then
\begin{equation}\label{maineq}
\sum\limits_{i=1}^n [(d_i- \ell)^2-d_i m_i+\lambda \ell+\mu(n- \ell)]x_i^2 \leq 0.\end{equation}
Moreover, the equality in (\ref{maineq}) holds if and only if for any distinct  vertices $i, j\in V$, the following two statements hold:
\begin{align}
ij\in E {\rm ~~~and~~~} x_i\not=x_j \quad &\Rightarrow \quad |G_1(i)\cap G_1(j)|= \lambda(G), \label{equ1}\\
ij\not\in E {\rm ~~~and~~~} x_i\not=x_j \quad &\Rightarrow \quad |G_1(i)\cap G_1(j)|= \mu(G). \label{equ2}
\end{align}
\end{thm}

\begin{proof}
Because $X$ is an eigenvector of $L(G)$ corresponding to $\ell$,
\begin{equation}\label{e1}
\Vert (D(G)-\ell I)X \Vert^2  = \Vert (D(G)-L(G))X \Vert^2= \Vert A(G)X \Vert^2.
\end{equation}
As the $ij$-entry of $A(G)^2$ is the number $w_{ij}$ of walks of length $2$ from $i$ to $j$ and noting that
$w_{ii}=d_i$, we have
\begin{align}
 \Vert A(G)X \Vert^2=& X^\top A(G)^2 X \\ =& \sum\limits_{i\in V}  d_{i} x_i^2 + 2 \sum\limits_{j<k} w_{jk}x_j x_k \nonumber\\
=& \sum\limits_{i\in V} d_{i} x_i^2 + \sum\limits_{j<k}w_{jk}(x_j^2+ x_k^2-(x_j-x_k)^2) \nonumber\\
=& \sum\limits_{i\in V}( \sum\limits_{ji\in E} d_{j})x_i^2) - \sum\limits_{\begin{subarray}{c} j<k \\ jk\in E\end{subarray}}w_{jk}(x_j-x_k)^2 -\sum\limits_{\begin{subarray}{c} j<k \\ jk\not\in E\end{subarray}}w_{jk}(x_j-x_k)^2. \label{e2}
\end{align}
As $\lambda\leq\lambda(G)=\min_{ij\in E}w_{ij}$ and   $\mu\leq\mu(G)=\min_{ij\not\in E} w_{ij}$ and by (\ref{e1}), (\ref{e2}),
we have
\begin{equation}\label{e3}\sum\limits_{i\in V} (d_i-\ell)^2x_i^2 \leq \sum\limits_{i\in V} d_i m_i x_i^2 - \lambda\sum\limits_{\begin{subarray}{c} j<k \\ jk\in E\end{subarray}}(x_j-x_k)^2-\mu\sum\limits_{\begin{subarray}{c} j<k \\ jk\not\in E\end{subarray}}(x_j-x_k)^2.\end{equation}
Applying (\ref{imp}) and that $n-\ell(G)$ is  eigenvalue of $L(G^c)$ with the same eigenvector $X$
 to (\ref{e3}), we have
\begin{align*}
\sum\limits_{i\in V} (d_i-\ell)^2x_i^2 \leq & \sum\limits_{i\in V} d_i m_i x_i^2 - \lambda X^\top L(G) X - \mu X^\top L(G^c) X \\
=& \sum\limits_{i\in V} d_i m_i x_i^2 - \lambda \ell\Vert X \Vert^2 - \mu(n- \ell)\Vert X \Vert^2 \\
= &\sum\limits_{i\in V} d_i m_i x_i^2 - \lambda \ell\sum\limits_{i\in V} x_i^2-\mu(n-  \ell)\sum\limits_{i=1}^n x_i^2,
\end{align*}
and (\ref{maineq}) immediately follows from this. Note that the equality holds in (\ref{maineq}) if and only if the equality in (\ref{e3}) holds, and this equivalent to (\ref{equ1}), (\ref{equ2}).
\end{proof}

The expression \begin{equation}\label{maineq2}
(d_i- \ell)^2-d_i m_i+\lambda \ell+\mu(n- \ell)=\ell^2-(2d_i-\lambda+\mu) \ell+(d_i^2-d_i m_i+\mu n) \end{equation}
 inside the summation in (\ref{maineq}) is a quadratic polynomial
in variable $\ell$ and is not positive  for some $i$.
Solving the quadratic polynomial,
we have  the following upper bound of Laplacian index $\ell_1(G)$, lower bound of the algebraic connectivity $\ell_{n-1}(G)$ and upper bound of Laplacian spread $\mathscr{S}_L(G)$ of $G$.

\begin{cor}\label{eig} Referring to the notations in Theorem \ref{mainthm}, the following three inequalities hold:
\begin{align*}\ell_1(G)\leq &\underset{i \in V}{\max}\left\{ \dfrac{2d_i-\lambda+\mu+\sqrt{4d_i m_i-4(\lambda-\mu)d_i+(\lambda-\mu)^2-4\mu n}}{2} \right\}, \\
\ell_{n-1}(G)\geq &\underset{i \in V}{\min}\left\{ \dfrac{2d_i-\lambda+\mu-\sqrt{4d_i m_i-4(\lambda-\mu)d_i+(\lambda-\mu)^2-4\mu n}}{2} \right\}  \\
\end{align*}
and
\begin{align*}
\mathscr{S}_L(G) \leq& \underset{i \in V}{\max}\left\{ \dfrac{2d_i-\lambda+\mu+\sqrt{4d_i m_i-4(\lambda-\mu)d_i+(\lambda-\mu)^2-4\mu n}}{2} \right\} \\
-&\underset{i\in V}{\min}\left\{ \dfrac{2d_i-\lambda+\mu-\sqrt{4d_i m_i-4(\lambda-\mu)d_i+(\lambda-\mu)^2-4\mu n}}{2} \right\},
\end{align*}
where the $\max$ and $\min$ are over vertices $i\in V$ and exclude the terms with  a negative term in the square root.
\qed
\end{cor}

Note that the above upper bound of $\ell_1(G)$ with $\lambda=0$ and $\mu=0$ is  \eqref{2004}, and with $\lambda=\lambda(G)$ and $\mu=0$ is previously given in \cite[Theorem 3.2]{2013}.

The following corollary is another application of (\ref{maineq}).
\begin{cor}\label{Ddelta}
Referring to the notations in Theorem \ref{mainthm}, the following three inequalities hold:
\begin{align}
\ell_1(G) \leq &~\dfrac{2\Delta-\lambda+\mu+\sqrt{(2\Delta-
\lambda+\mu)^2-4\mu n}}{2}, \label{ee3}\\
\ell_{n-1}(G)\geq &~ \dfrac{2\delta-\lambda+\mu-\sqrt{(2\delta-
\lambda+\mu)^2-4\mu n -4 \delta^2+4 \Delta^2}}{2}, \label{ee5}
\end{align}
and
\begin{align}
\mathscr{S}_L(G) \leq &~\Delta-\delta+\frac{1}{2}\left[ \sqrt{(2\Delta-\lambda+\mu)^2-4\mu n}\right. \nonumber\\
                      &~~~~~~~~~~~~~~~\left.+\sqrt{(2\delta-
\lambda+\mu)^2-4\mu n -4 \delta^2+4 \Delta^2}\right], \label{ee7}
\end{align}
where $\delta$ and $\Delta$ are the maximum degree and the minimum degree in $G$.
Moreover if $G$ is $k$-regular then
\begin{equation}
\mathscr{S}_L(G) \leq \sqrt{(2k-\lambda+\mu)^2-4\mu n}. \label{ee8}\end{equation}
\end{cor}

\begin{proof}
By \eqref{maineq} with $\ell=\ell_1(G)$ there exists $i\in V$ such that the term in \eqref{maineq2} is not positive.
Using $\ell_1(G)\geq \Delta+1 > d_i,$ and $\Delta\geq m_i$, we have
$$ (\Delta- \ell_1(G))^2-\Delta^2+\lambda \ell_1(G)+\mu(n- \ell_1(G))  \leq (d_i- \ell_1(G))^2-d_i m_i+\lambda \ell_1(G)+\mu(n- \ell_1(G))\leq 0.$$
Solving the above quadratic inequality on the left for $\ell_1(G)$,
we have (\ref{ee3}). Similarly, by considering
$\ell=\ell_{n-1}(G)$  in \eqref{maineq}, there exists $j\in V$ such that  \eqref{maineq2}  with $i=j$ is not positive.
Using $\ell_{n-1}(G)\leq \delta \leq d_j$ and $\Delta\geq m_i$, we have
$$ (\delta- \ell_{n-1}(G))^2-\Delta^2+\lambda \ell_{n-1}(G)+\mu(n- \ell_{n-1}(G))  \leq (d_j- \ell_{n-1}(G))^2-d_j m_j+\lambda \ell_{n-1}(G)+\mu(n- \ell_{n-1}(G))\leq 0.$$
Solving the quadratic inequality on the left for $\ell_{n-1}(G)$,
we have (\ref{ee5}). The line (\ref{ee7}) is immediate from (\ref{ee3}), (\ref{ee5}), and (\ref{ee8}) is from (\ref{ee7}).
\end{proof}
Next we prove that the strongly regular graphs satisfy all the above   equalities.
\begin{cor}
If $G$ is a strongly regular graph with parameters $(n,k,\lambda(G),\mu(G))$, then $k=\delta=\Delta$ and  the equality in (\ref{maineq}), the three  equalities in Corollary~\ref{eig} and the three equalities (\ref{ee3}), (\ref{ee5}), (\ref{ee7}) all hold for $\lambda=\lambda(G)$ and $\mu=\mu(G)$.
\end{cor}
\begin{proof}
This is clear since (\ref{equ1}) and (\ref{equ2}) hold in a strongly regular graph.
\end{proof}

\section{Other extremal graphs}\label{s4}

A graph is  {\it extremal} for an inequality holding for graphs if the graph attains the equality. In this section, we shall provide extremal graphs for inequalities mentioned in the previous section, excluding strongly regular graphs. Throughout this section we assume $\lambda=\lambda(G)$ and $\mu=\mu(G).$
Let
$$\alpha_1(G)=\underset{i \in V}{\max}\left\{ \dfrac{2d_i-\lambda+\mu+\sqrt{4d_i m_i-4(\lambda-\mu)d_i+(\lambda-\mu)^2-4\mu n}}{2} \right\}$$ $$\left( \text{resp. }\beta_1(G)=\underset{i \in V}{\min}\left\{ \dfrac{2d_i-\lambda+\mu-\sqrt{4d_i m_i-4(\lambda-\mu)d_i+(\lambda-\mu)^2-4\mu n}}{2} \right\}\right)$$
denote the upper bound (resp. lower bound) of Laplacian index $\ell_1(G)$  (resp. algebraic connectivity $\ell_{n-1}(G)$) described in Corollary \ref{eig} and let
$$\alpha_2(G)=\dfrac{2\Delta-\lambda+\mu+\sqrt{(2\Delta-
\lambda+\mu)^2-4\mu n}}{2}$$
$$\left( \text{resp. }\beta_2(G)= \dfrac{2\delta-\lambda+\mu-\sqrt{(2\delta-
\lambda+\mu)^2-4\mu n -4 \delta^2+4 \Delta^2}}{2}\right)$$
denote the upper bound (resp. lower bound) of Laplacian index $\ell_1(G)$ (resp. algebraic connectivity $\ell_{n-1}(G)$) described in Corollary \ref{Ddelta}.

\begin{exam} The graph $G=X_8$ depicted on the left of Figure~\ref{pi1}  is $3$-regular of order $8$ with $\lambda(X_8)=0,$ $\mu(X_8)=1$ and $\ell_1(X_8), \ell_{7}(X_8)=(7\pm \sqrt{17})/2=\ell, \mathscr{S}_L(G)= \sqrt{17}$, so is extremal for (\ref{maineq}). Note that $\alpha_1(X_8)=\alpha_2(X_8)=(7+ \sqrt{17})/2$ and $\beta_1(X_8)=\beta_2(X_8)=(7- \sqrt{17})/2$. Hence $X_8$ is extremal for the three inequalities in Corollary~\ref{eig}, (\ref{ee3}), (\ref{ee5}) and (\ref{ee7}). On the other hand, the complement graph $G^c=X_8^c$ of $X_8$ has  $\lambda(X_8^c)=1,$ $\mu(X_8^c)=2$ and $\ell_1(X_8^c), \ell_{7}(X_8^c)=(9\pm \sqrt{17})/2=\ell, \mathscr{S}_L(X_8^c)= \sqrt{17}$, so is also extremal for (\ref{maineq}). Note that $\alpha_1(X_8^c)=\alpha_2(X_8^c)=(9+ \sqrt{17})/2$ and $\beta_1(X_8^c)=\beta_2(X_8^c)=(9- \sqrt{17})/2$. Hence $X_8^c$ is extremal for the three inequalities in Corollary~\ref{eig}, (\ref{ee3}), (\ref{ee5}) and (\ref{ee7}).
\end{exam}

\bigskip

\begin{figure}[htb]
\begin{center}
$\begin{array}{ccc}
\begin{tikzpicture}[scale=0.9, line cap=round,line join=round,>=triangle 45,x=1.0cm,y=1.0cm]
\draw (4.79507888805,1.79009693206)-- (3.4436255862,1.13927132338);
\draw (3.4436255862,1.13927132338)-- (6.14653218991,1.13927132338);
\draw (6.14653218991,1.13927132338)-- (4.79507888805,1.79009693206);
\draw (3.10984418527,-0.323120544893)-- (4.04507888805,-1.4958677686);
\draw (4.04507888805,-1.4958677686)-- (4.79507888805,0.0370999248685);
\draw (4.79507888805,0.0370999248685)-- (5.54507888805,-1.4958677686);
\draw (5.54507888805,-1.4958677686)-- (6.48031359084,-0.323120544893);
\draw (6.48031359084,-0.323120544893)-- (4.04507888805,-1.4958677686);
\draw (3.10984418527,-0.323120544893)-- (5.54507888805,-1.4958677686);
\draw (3.4436255862,1.13927132338)-- (3.10984418527,-0.323120544893);
\draw (4.79507888805,1.79009693206)-- (4.79507888805,0.0370999248685);
\draw (6.14653218991,1.13927132338)-- (6.48031359084,-0.323120544893);

\draw [fill=black] (4.04507888805,-1.4958677686) circle (1.5pt);
\draw[color=black] (4.0448985725,-1.75) node {$7$};
\draw [fill=black] (5.54507888805,-1.4958677686) circle (1.5pt);
\draw[color=black] (5.53250187829,-1.75) node {$8$};
\draw [fill=black] (6.48031359084,-0.323120544893) circle (1.5pt);
\draw[color=black] (6.67450037566,-0.263425995492) node {$6$};
\draw [fill=black] (6.14653218991,1.13927132338) circle (1.5pt);
\draw[color=black] (6.35894815928,1.25422990233) node {$3$};
\draw [fill=black] (4.79507888805,1.79009693206) circle (1.5pt);
\draw[color=black] (4.7962133734,2.11072877536) node {$1$};
\draw [fill=black] (3.4436255862,1.13927132338) circle (1.5pt);
\draw[color=black] (3.20342599549,1.25422990233) node {$2$};
\draw [fill=black] (3.10984418527,-0.323120544893) circle (1.5pt);
\draw[color=black] (2.84279489106,-0.248399699474) node {$4$};
\draw [fill=black] (4.79507888805,0.0370999248685) circle (1.5pt);
\draw[color=black] (5.00658151766,0.112231404959) node {$5$};
\end{tikzpicture} &~~~~~~~~~~~~~~~~~~&
\begin{tikzpicture}[scale=0.9, line cap=round,line join=round,>=triangle 45,x=1.0cm,y=1.0cm]
\draw (4.79507888805,1.79009693206)-- (3.10984418527,-0.323120544893);
\draw (4.79507888805,1.79009693206)-- (4.04507888805,-1.4958677686);
\draw (4.79507888805,1.79009693206)-- (5.54507888805,-1.4958677686);
\draw (4.79507888805,1.79009693206)-- (6.48031359084,-0.323120544893);
\draw (3.4436255862,1.13927132338)-- (4.04507888805,-1.4958677686);
\draw (3.4436255862,1.13927132338)-- (6.48031359084,-0.323120544893);
\draw (3.4436255862,1.13927132338)-- (4.79507888805,0.0370999248685);
\draw (3.10984418527,-0.323120544893)-- (4.79507888805,0.0370999248685);
\draw (3.10984418527,-0.323120544893)-- (6.48031359084,-0.323120544893);
\draw (3.10984418527,-0.323120544893)-- (6.14653218991,1.13927132338);
\draw (4.04507888805,-1.4958677686)-- (6.14653218991,1.13927132338);
\draw (4.04507888805,-1.4958677686)-- (5.54507888805,-1.4958677686);
\draw (5.54507888805,-1.4958677686)-- (6.14653218991,1.13927132338);
\draw (5.54507888805,-1.4958677686)-- (3.4436255862,1.13927132338);
\draw (6.48031359084,-0.323120544893)-- (4.79507888805,0.0370999248685);
\draw (4.79507888805,0.0370999248685)-- (6.14653218991,1.13927132338);

\draw [fill=black] (4.04507888805,-1.4958677686) circle (1.5pt);
\draw[color=black] (4.0448985725,-1.75) node {$7$};
\draw [fill=black] (5.54507888805,-1.4958677686) circle (1.5pt);
\draw[color=black] (5.53250187829,-1.75) node {$8$};
\draw [fill=black] (6.48031359084,-0.323120544893) circle (1.5pt);
\draw[color=black] (6.67450037566,-0.263425995492) node {$6$};
\draw [fill=black] (6.14653218991,1.13927132338) circle (1.5pt);
\draw[color=black] (6.35894815928,1.25422990233) node {$3$};
\draw [fill=black] (4.79507888805,1.79009693206) circle (1.5pt);
\draw[color=black] (4.7962133734,2.11072877536) node {$1$};
\draw [fill=black] (3.4436255862,1.13927132338) circle (1.5pt);
\draw[color=black] (3.20342599549,1.25422990233) node {$2$};
\draw [fill=black] (3.10984418527,-0.323120544893) circle (1.5pt);
\draw[color=black] (2.84279489106,-0.248399699474) node {$4$};
\draw [fill=black] (4.79507888805,0.0370999248685) circle (1.5pt);
\draw[color=black] (4.84,-0.167) node {$5$};
\end{tikzpicture}
\end{array}$
\end{center}
\caption{The graph $X_8$ on the left has $\lambda(X_8)=0,$ $\mu(X_8)=1$  and its complement graph $X_8^c$ on the right has $\lambda(X_8^c)=1,$ $\mu(X_8^c)=2$.}\label{pi1}
\end{figure}
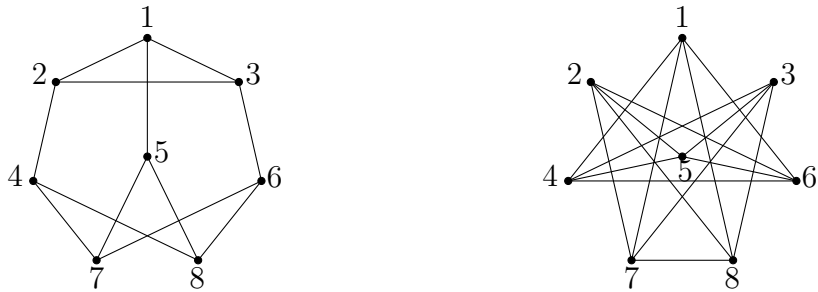

\begin{exam}\label{Y_8}
The graph $G=Y_8$ depicted on the left of Figure~\ref{YZ} is obtained from the complete graph $K_8$ of order $8$   by deleting two vertex disjoint cycles of order $4$. It is $5$-regular of order $8$ with $\lambda(Y_8)=2$, $\mu(Y_8)=4,$ $\ell_1(Y_8)=8$ and $\ell_7(Y_8)=4,$ so is extremal for (\ref{maineq}) with $\ell=8,4$ respectively. Note that $\alpha_1(Y_8)=\alpha_2(Y_8)=8$, $\beta_1(Y_8)=\beta_2(Y_8)=4$. Hence $G$ is extremal for for the three inequalities in Corollary~\ref{eig}, (\ref{ee3}), (\ref{ee5}) and (\ref{ee7}).
\end{exam}

\bigskip

\begin{figure}[htb]
\begin{center}
$\begin{array}{ccc}
\begin{tikzpicture}[scale=0.8, line cap=round,line join=round,>=triangle 45,x=1.0cm,y=1.0cm]
\draw [line width=0.4pt,dotted] (0.,4.62132034356)-- (1.5,4.62132034356);
\draw [line width=0.4pt,dotted] (2.56066017178,3.56066017178)-- (2.56066017178,2.06066017178);
\draw [line width=0.4pt,dotted] (2.56066017178,2.06066017178)-- (1.5,1.);
\draw (1.5,1.)-- (2.56066017178,3.56066017178);
\draw (0.,1.)-- (-1.06066017178,2.06066017178);
\draw [line width=0.4pt,dotted] (-1.06066017178,2.06066017178)-- (-1.06066017178,3.56066017178);
\draw (-1.06066017178,3.56066017178)-- (0.,1.);
\draw (0.,4.62132034356)-- (2.56066017178,3.56066017178);
\draw (0.,4.62132034356)-- (2.56066017178,2.06066017178);
\draw (0.,4.62132034356)-- (1.5,1.);
\draw (1.5,4.62132034356)-- (-1.06066017178,3.56066017178);
\draw [line width=0.4pt,dotted] (1.5,4.62132034356)-- (-1.06066017178,2.06066017178);
\draw (1.5,4.62132034356)-- (0.,1.);
\draw (-1.06066017178,3.56066017178)-- (2.56066017178,3.56066017178);
\draw (2.56066017178,2.06066017178)-- (-1.06066017178,2.06066017178);
\draw [line width=0.4pt,dotted] (0.,1.)-- (1.5,1.);
\draw (0.,4.62132034356)-- (-1.06066017178,2.06066017178);
\draw [line width=0.4pt,dotted] (-1.06066017178,3.56066017178)-- (0.,4.62132034356);
\draw (1.5,4.62132034356)-- (2.56066017178,3.56066017178);
\draw [line width=0.4pt,dotted] (2.56066017178,3.56066017178)-- (0.,1.);
\draw (0.,1.)-- (2.56066017178,2.06066017178);
\draw (-1.06066017178,2.06066017178)-- (2.56066017178,3.56066017178);
\draw (-1.06066017178,2.06066017178)-- (1.5,1.);
\draw (0.,1.)-- (0.,4.62132034356);
\draw (1.5,1.)-- (-1.06066017178,3.56066017178);
\draw (1.5,1.)-- (1.5,4.62132034356);
\draw (2.56066017178,2.06066017178)-- (-1.06066017178,3.56066017178);
\draw (2.56066017178,2.06066017178)-- (1.5,4.62132034356);

\draw [fill=black] (0.,1.) circle (1.5pt);
\draw[color=black] (-0.06,0.76) node {6};
\draw [fill=black] (1.5,1.) circle (1.5pt);
\draw[color=black] (1.48,0.74) node {5};
\draw [fill=black] (2.56066017178,2.06066017178) circle (1.5pt);
\draw[color=black] (2.82,2.2) node {4};
\draw [fill=black] (2.56066017178,3.56066017178) circle (1.5pt);
\draw[color=black] (2.9,3.68) node {3};
\draw [fill=black] (1.5,4.62132034356) circle (1.5pt);
\draw[color=black] (1.54,5.) node {2};
\draw [fill=black] (0.,4.62132034356) circle (1.5pt);
\draw[color=black] (0.04,4.96) node {1};
\draw [fill=black] (-1.06066017178,3.56066017178) circle (1.5pt);
\draw[color=black] (-1.4,3.64) node {8};
\draw [fill=black] (-1.06066017178,2.06066017178) circle (1.5pt);
\draw[color=black] (-1.38,2.22) node {7};
\end{tikzpicture}&~~~~~~~~~~~~~~~~~~&
\begin{tikzpicture}[scale=0.8, line cap=round,line join=round,>=triangle 45,x=1.0cm,y=1.0cm]
\draw [line width=0.4pt,dotted] (0.,4.62132034356)-- (1.5,4.62132034356);
\draw [line width=0.4pt,dotted] (2.56066017178,3.56066017178)-- (2.56066017178,2.06066017178);
\draw [line width=0.4pt,dotted] (2.56066017178,2.06066017178)-- (1.5,1.);
\draw [line width=0.4pt,dotted] (1.5,1.)-- (2.56066017178,3.56066017178);
\draw [line width=0.4pt,dotted] (0.,1.)-- (-1.06066017178,2.06066017178);
\draw [line width=0.4pt,dotted] (-1.06066017178,2.06066017178)-- (-1.06066017178,3.56066017178);
\draw (-1.06066017178,3.56066017178)-- (0.,1.);
\draw (0.,4.62132034356)-- (2.56066017178,3.56066017178);
\draw (0.,4.62132034356)-- (2.56066017178,2.06066017178);
\draw (0.,4.62132034356)-- (1.5,1.);
\draw (1.5,4.62132034356)-- (-1.06066017178,3.56066017178);
\draw (1.5,4.62132034356)-- (-1.06066017178,2.06066017178);
\draw [line width=0.4pt,dotted] (1.5,4.62132034356)-- (0.,1.);
\draw (-1.06066017178,3.56066017178)-- (2.56066017178,3.56066017178);
\draw (2.56066017178,2.06066017178)-- (-1.06066017178,2.06066017178);
\draw (0.,1.)-- (1.5,1.);
\draw (0.,4.62132034356)-- (-1.06066017178,2.06066017178);
\draw [line width=0.4pt,dotted] (-1.06066017178,3.56066017178)-- (0.,4.62132034356);
\draw (1.5,4.62132034356)-- (2.56066017178,3.56066017178);
\draw (2.56066017178,3.56066017178)-- (0.,1.);
\draw (0.,1.)-- (2.56066017178,2.06066017178);
\draw (-1.06066017178,2.06066017178)-- (2.56066017178,3.56066017178);
\draw (-1.06066017178,2.06066017178)-- (1.5,1.);
\draw (0.,1.)-- (0.,4.62132034356);
\draw (1.5,1.)-- (-1.06066017178,3.56066017178);
\draw (1.5,1.)-- (1.5,4.62132034356);
\draw (2.56066017178,2.06066017178)-- (-1.06066017178,3.56066017178);
\draw (2.56066017178,2.06066017178)-- (1.5,4.62132034356);

\draw [fill=black] (0.,1.) circle (1.5pt);
\draw[color=black] (-0.06,0.76) node {6};
\draw [fill=black] (1.5,1.) circle (1.5pt);
\draw[color=black] (1.48,0.74) node {5};
\draw [fill=black] (2.56066017178,2.06066017178) circle (1.5pt);
\draw[color=black] (2.82,2.2) node {4};
\draw [fill=black] (2.56066017178,3.56066017178) circle (1.5pt);
\draw[color=black] (2.9,3.68) node {3};
\draw [fill=black] (1.5,4.62132034356) circle (1.5pt);
\draw[color=black] (1.54,5.) node {2};
\draw [fill=black] (0.,4.62132034356) circle (1.5pt);
\draw[color=black] (0.04,4.96) node {1};
\draw [fill=black] (-1.06066017178,3.56066017178) circle (1.5pt);
\draw[color=black] (-1.4,3.64) node {8};
\draw [fill=black] (-1.06066017178,2.06066017178) circle (1.5pt);
\draw[color=black] (-1.38,2.22) node {7};
\end{tikzpicture}
\end{array}$
\end{center}
\caption{The graph $Y_8$ on the left has $\lambda(Y_8)=2,$ $\mu(Y_8)=4$,  and the graph $Z_8$ on the right has $\lambda(Z_8)=2,$ $\mu(Z_8)=4$.
}\label{YZ}
\end{figure}
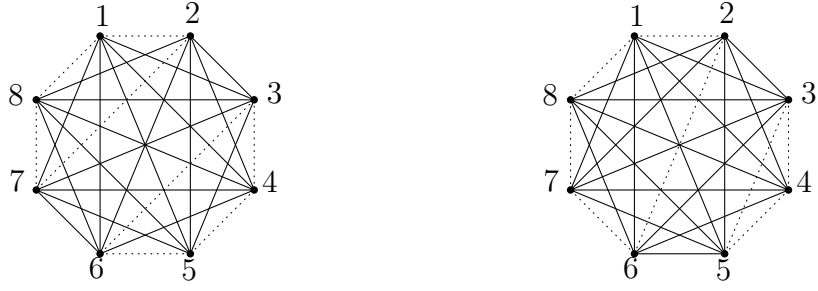

\begin{exam}\label{Z_8}
The graph $G=Z_8$ depicted on the right of Figure~\ref{YZ} is obtained from $K_8$ by deleting two vertex disjoint cycles of order $3$ and $5$ respectively. It is $5$-regular of order $8$ with $\lambda(Z_8)=2$, $\mu(Z_8)=4,$ $\ell_1(Z_8)=8$ and  so is extremal for (\ref{maineq}) with $\ell=8$. Note that $\alpha_1(Z_8)=\alpha_2(Z_8)=8$, $\beta_1(Z_8)=\beta_2(Z_8)=4$ and $\ell_7(Z_8)=(11-\sqrt{5})/2$. Hence $Z_8$ is extremal for the first inequality in Corollary~\ref{eig} and (\ref{ee3}) and is not extremal for other inequalities.
\end{exam}

\begin{exam}
The graph $G=U_8$ depicted on the left of Figure~\ref{U} is $4$-regular of order $8$ with $\lambda(U_8)=0$, $\mu(U_8)=2,$ $\ell_7(U_8)=2,$ so is extremal for
(\ref{maineq}) with $\ell=2$. Note that $\alpha_1(U_8)=\alpha_2(U_8)=8$,  $\beta_1(U_8)=\beta_2(U_8)=2$ and $\ell_1(U_8)=6$. Hence $U_8$ is extremal for  the second inequality in Corollary~\ref{eig} and (\ref{ee5}), and is not extremal for other inequalities.
On the other hand, the complement graph $G=U_8^c$ of $U_8$ depicted on the right of Figure~\ref{U} is $3$-regular of order $8$ with $\lambda(U_8^c)=\mu(U_8^c)=0,$ $\ell_1(U_8^c)=6,$ so is extremal for
(\ref{maineq}) with $\ell=6$. Note that $\alpha_1(U_8^c)=\alpha_2(U_8^c)=6$,  $\beta_1(U_8^c)=\beta_2(U_8^c)=0$ and $\ell_7(U_8^c)=2$.  Hence $U_8^c$ is extremal for  the first inequality in Corollary~\ref{eig} and (\ref{ee3}), and is not extremal for other inequalities.
\end{exam}

\bigskip

\begin{figure}[htb]
\begin{center}
$\begin{array}{ccc}

\begin{tikzpicture}[scale=0.8, line cap=round,line join=round,>=triangle 45,x=1.0cm,y=1.0cm]
\draw (0,4.62)-- (1.5,4.62);
\draw (2.56,3.56)-- (2.56,2.06);
\draw (2.56,2.06)-- (1.5,1);
\draw (1.5,1)-- (2.56,3.56);
\draw (0,1)-- (-1.06,2.06);
\draw (-1.06,2.06)-- (-1.06,3.56);
\draw (-1.06,3.56)-- (0,1);
\draw (0,4.62)-- (2.56,3.56);
\draw (0,4.62)-- (2.56,2.06);
\draw (0,4.62)-- (1.5,1);
\draw (1.5,4.62)-- (-1.06,3.56);
\draw (1.5,4.62)-- (-1.06,2.06);
\draw (1.5,4.62)-- (0,1);
\draw (-1.06,3.56)-- (2.56,3.56);
\draw (2.56,2.06)-- (-1.06,2.06);
\draw (0,1)-- (1.5,1);
\fill [color=black] (0,1) circle (1.5pt);
\draw[color=black] (-0.04,0.78) node {6};
\fill [color=black] (1.5,1) circle (1.5pt);
\draw[color=black] (1.5,0.76) node {5};
\fill [color=black] (2.56,2.06) circle (1.5pt);
\draw[color=black] (2.86,2.22) node {4};
\fill [color=black] (2.56,3.56) circle (1.5pt);
\draw[color=black] (2.94,3.7) node {3};
\fill [color=black] (1.5,4.62) circle (1.5pt);
\draw[color=black] (1.56,5.02) node {2};
\fill [color=black] (0,4.62) circle (1.5pt);
\draw[color=black] (0.06,4.98) node {1};
\fill [color=black] (-1.06,3.56) circle (1.5pt);
\draw[color=black] (-1.36,3.66) node {8};
\fill [color=black] (-1.06,2.06) circle (1.5pt);
\draw[color=black] (-1.34,2.24) node {7};
\end{tikzpicture}&~~~~~~~~~~~~~~~~~~&
\begin{tikzpicture}[scale=0.8, line cap=round,line join=round,>=triangle 45,x=1.0cm,y=1.0cm]
\draw (0.,4.62132034356)-- (-1.06066017178,3.56066017178);
\draw (0.,4.62132034356)-- (-1.06066017178,2.06066017178);
\draw (0.,4.62132034356)-- (0.,1.);
\draw (-1.06066017178,3.56066017178)-- (2.56066017178,2.06066017178);
\draw (-1.06066017178,3.56066017178)-- (1.5,1.);
\draw (-1.06066017178,2.06066017178)-- (2.56066017178,3.56066017178);
\draw (-1.06066017178,2.06066017178)-- (1.5,1.);
\draw (0.,1.)-- (2.56066017178,3.56066017178);
\draw (0.,1.)-- (2.56066017178,2.06066017178);
\draw (1.5,1.)-- (1.5,4.62132034356);
\draw (2.56066017178,2.06066017178)-- (1.5,4.62132034356);
\draw (2.56066017178,3.56066017178)-- (1.5,4.62132034356);

\fill [color=black] (0,1) circle (1.5pt);
\draw[color=black] (-0.04,0.78) node {6};
\fill [color=black] (1.5,1) circle (1.5pt);
\draw[color=black] (1.5,0.76) node {5};
\fill [color=black] (2.56,2.06) circle (1.5pt);
\draw[color=black] (2.86,2.22) node {4};
\fill [color=black] (2.56,3.56) circle (1.5pt);
\draw[color=black] (2.94,3.7) node {3};
\fill [color=black] (1.5,4.62) circle (1.5pt);
\draw[color=black] (1.56,5.02) node {2};
\fill [color=black] (0,4.62) circle (1.5pt);
\draw[color=black] (0.06,4.98) node {1};
\fill [color=black] (-1.06,3.56) circle (1.5pt);
\draw[color=black] (-1.36,3.66) node {8};
\fill [color=black] (-1.06,2.06) circle (1.5pt);
\draw[color=black] (-1.34,2.24) node {7};

\end{tikzpicture}
\end{array}$
\end{center}
\caption{The  graph $U_8$  on the left  has $\lambda(U_8)=0,$ $\mu(U_8)=2$, and its complement graph $U_8^c$ on the right has $\lambda(U_8^c)=\mu(U_8^c)=0$.
}\label{U}
\end{figure}
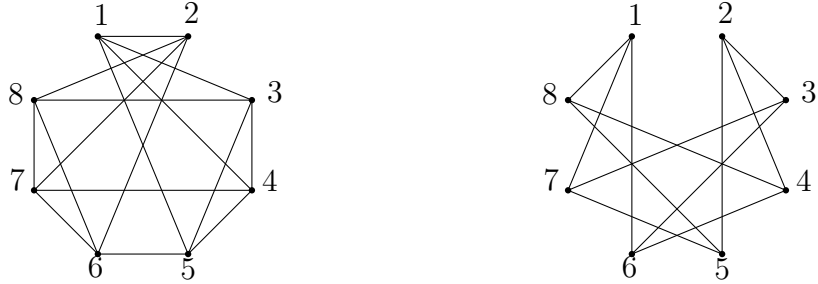
\bigskip

\begin{exam}
Let $G=K_{a, b}$ be the complete bipartite graph of bipartition orders $a$ and $b,$ respectively, where $a<b$  and $n=a+b$. Then $\lambda(K_{a, b})=0$, $\mu(K_{a, b})=a,$ $(d_i,m_i)=(a,b)$ or $(b,a),$ so the lower bound of $\ell_{n-1}(K_{a, b})$ in the second inequality of Corollary~\ref{eig} is $\beta_1(K_{a, b})=\min\{a,(2b+a-\sqrt{a(4b-3a)})/2\}=a$. Also $\ell_{n-1}(K_{a, b})=a$ since $\ell_{n-1}(K_{a, b})\leq \kappa(K_{a, b}) \leq \delta=a.$ Hence $K_{a, b}$ is extremal for (\ref{maineq}) with $\ell=\ell_{n-1}(K_{a, b})$ and the second inequality in Corollary~\ref{eig}. Note that $\alpha_1(K_{a, b})=(2b+a+\sqrt{a(4b-3a)})/2,$ $\alpha_2(K_{a, b})=(2b+a+\sqrt{4b^2-3a^2})/2,$  $\beta_2(K_{a, b})=2a-b$ and $\ell_1(K_{a, b})=a+b$. Hence $K_{a, b}$ is not extremal for other inequalities.
\end{exam}

Next we provide a graph which is extremal only for (\ref{maineq}) about $\ell=\ell_1(G)$.
\begin{exam}
Let $G=F_t$ $(t > 1)$ be a fan of order $n=2t+1$ as depicted in  Figure~\ref{8_4}. Then $\lambda(F_t)=1$, $\mu(F_t)=1,$ $(d_i,m_i)=(2,t+1)$ or $(2t,2),$ and $X=(\mathbf{1}_{2t}^\top,-2t)^\top$ is an eigenvector corresponding to the eigenvalue $\ell_1(F_t)=2t+1$, where $\mathbf{1}$ is all one vector.
 One can check that $F_t$ is extremal for (\ref{maineq}) about $\ell=\ell_1(F_t)$.  On the other hand,  $\alpha_1(F_t)=(2t+\sqrt{2t-1})$, $\alpha_2(F_t)=2t+\sqrt{4t^2-2t-1}$,
 $\beta_1(F_t)=\min\{2t-\sqrt{2t-1},1\}=1\leq \ell_{2t}(F_t) \leq \kappa(F_t)=1$, $\beta_2(F_t)=2-\sqrt{4t^2-2t-1}$, so $F_t$ is also extremal for (\ref{maineq}) with $\ell=\ell_{2t}(F_t)$, the second inequality in Corollary~\ref{eig}, but is not extremal for any other inequalities.
\end{exam}

\begin{figure}[ht]
\begin{center}
\begin{tikzpicture}[scale=0.7, line cap=round,line join=round,>=triangle 45,x=1.0cm,y=1.0cm]
\draw (8.23054678984,-2.31496327933)-- (8.68,-5.1);
\draw (8.68,-5.1)-- (7.58,-5.06);
\draw (7.58,-5.06)-- (8.23054678984,-2.31496327933);
\draw (8.23054678984,-2.31496327933)-- (5.82950666347,-3.79599131255);
\draw (5.82950666347,-3.79599131255)-- (5.44551006917,-2.76441648949);
\draw (5.44551006917,-2.76441648949)-- (8.23054678984,-2.31496327933);
\draw (8.23054678984,-2.31496327933)-- (5.94341702608,-0.66345634102);
\draw (5.94341702608,-0.66345634102)-- (6.74951875663,0.0860768470373);
\draw (6.74951875663,0.0860768470373)-- (8.23054678984,-2.31496327933);
\draw (8.23054678984,-2.31496327933)-- (8.88109357969,0.430073441338);
\draw (8.88109357969,0.430073441338)-- (9.88205372815,-0.0278335155636);
\draw (9.88205372815,-0.0278335155636)-- (8.23054678984,-2.31496327933);
\draw (8.23054678984,-2.31496327933)-- (10.5176765536,-3.96647021764);
\draw (10.5176765536,-3.96647021764)-- (10.9755835105,-2.96551006917);
\draw (10.9755835105,-2.96551006917)-- (8.23054678984,-2.31496327933);
\draw [fill=black] (7.58,-5.06) circle (1.5pt);
\draw[color=black] (7.48,-5.36) node {$2$};
\draw [fill=black] (8.68,-5.1) circle (1.5pt);
\draw[color=black] (8.72,-5.42) node {$1$};
\draw [fill=black] (10.5176765536,-3.96647021764) circle (1.5pt);
\draw[color=black] (10.9,-4.02) node {$2t$};
\draw [fill=black] (10.9755835105,-2.96551006917) circle (1.5pt);
\draw[color=black] (11.7,-2.88) node {$2t-1$};
\draw[color=black] (10.8,-1.35) node {\Large$\vdots$};
\draw [fill=black] (9.88205372815,-0.0278335155636) circle (1.5pt);
\draw[color=black] (10.14,0.18) node {$8$};
\draw [fill=black] (8.88109357969,0.430073441338) circle (1.5pt);
\draw[color=black] (8.94,0.82) node {$7$};
\draw [fill=black] (6.74951875663,0.0860768470373) circle (1.5pt);
\draw[color=black] (6.7,0.48) node {$6$};
\draw [fill=black] (5.94341702608,-0.66345634102) circle (1.5pt);
\draw[color=black] (5.6,-0.56) node {$5$};
\draw [fill=black] (5.44551006917,-2.76441648949) circle (1.5pt);
\draw[color=black] (4.96,-2.72) node {$4$};
\draw [fill=black] (5.82950666347,-3.79599131255) circle (1.5pt);
\draw[color=black] (5.38,-3.84) node {$3$};
\draw [fill=black] (8.23054678984,-2.31496327933) circle (1.5pt);
\draw[color=black] (9.1,-2.11) node {$2t+1$};
\end{tikzpicture}
\end{center}
\caption{The fan graph $F_{t}$ of order $2t+1$ with $\lambda(F_t)=1,$ $\mu(F_t)=1.$}\label{8_4}
\end{figure}
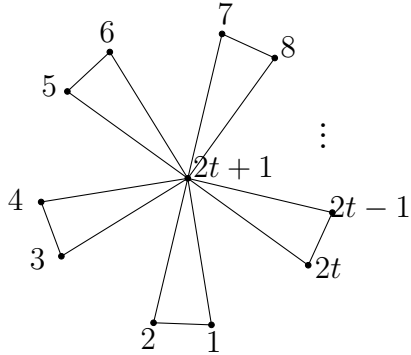

The following table summarizes the extremal graphs mentioned in this section which are not strongly regular.

\begin{table}[ht]
$$
\begin{array}{l|ccc|ccc}
\text{Graph} & \ell_1(G) & \alpha_1(G) & \alpha_2(G) & \ell_{n-1}(G) & \beta_1(G) & \beta_2(G) \\\hline
X_{8} & (7+\sqrt{17})/2 &(7+\sqrt{17})/2 & (7+\sqrt{17})/2  & (7-\sqrt{17})/2& (7-\sqrt{17})/2 & (7-\sqrt{17})/2\\
X_8^c & (9+ \sqrt{17})/2 &  (9+ \sqrt{17})/2 & (9+ \sqrt{17})/2  & (9- \sqrt{17})/2   & (9- \sqrt{17})/2 & (9- \sqrt{17})/2 \\
Y_{8} & 8 & 8 & 8  & 4 & 4 & 4 \\
Z_{8} & 8 & 8 & 8  & (11-\sqrt{5})/2 & 4 & 4 \\
U_{8} & 6 & 8 & 8  & 2 & 2 & 2 \\
U_{8}^c & 6 & 6 & 6  & 2 & 0 & 0 \\
K_{a,b}  & a+b &  x & x & a & a & 2a-b \\
F_{t}  & 2t+1 & y & z  & 1 & 1 & w \\
\hline
\end{array}
$$
$$a<b, t>1, x=\dfrac{2b+a+\sqrt{a(4b-3a)}}{2}, y=2t+\sqrt{2t-1}, z=2t+\sqrt{4t^2-2t-1}, w=2-\sqrt{4t^2-2t-1}.$$
\caption{Extremal graphs that are not strongly regular.}\label{T1}
\end{table}

\section{The $(n-3)$-regular graphs of order $n$} \label{s5}

From now on let $G$ denote an $(n-3)$-regular graph $G$ of order $n.$
Note that $G$ is obtained from the complete graph $K_n$ by deleting
some edges whose union forms  vertex disjoint cycles of order $n$. Denote $G$ in notation $G=K_n-(C_{n_1}\cup C_{n_2}\cup \cdots \cup C_{n_t}),$ where $n_1\geq  n_2\geq  \ldots\geq  n_t\geq 3$ is a nonincreasing integer sequence satisfying $n=n_1+n_2+\cdots+n_t.$ For example, $Y_8=K_8-2C_4$ in Example~\ref{Y_8} and $Z_8=K_8-(C_3\cup C_5)$ in Example~\ref{Z_8}. Note that $G$ is connected iff $n\geq 5.$
If $n=5$ then $G=K_5-C_5=C_5$ is a cycle of order $5$ and is a strongly regular graph with $\ell_1(G), \ell_4(G)=(5\pm\sqrt{5})/2.$ Hence we assume $n\geq 6.$

\begin{prop}\label{n-3}
The $(n-3)$-regular graph $G=K_n-(C_{n_1}\cup C_{n_2}\cup \cdots \cup C_{n_t})$
has $\lambda(G)=n-6$ and
$$\mu(G)= \left\{
           \begin{array}{ll}
             n-3, & \hbox{if $n_i=3$ for all $i$;} \\
             n-4, & \hbox{otherwise.}
           \end{array}
         \right.
$$
Moreover $G$ is strongly regular iff $n_i=3$ for all $1\leq i\leq t,$ and in this case 
$\ell_1(G)=n$ and $\ell_{n-1}(G)=n-3.$  
\end{prop}
\begin{proof}
As $G$ is $(n-3)$-regular,  $|G_1(i)\cap G_1(j)|\geq n-6$ for $i, j\in V.$
To prove $\lambda(G)=n-6$, we need to choose $ij\in E$ such that in the deleting cycles
that $i, j$ belong, two neighbors of $i$ do not overlap the two neighbors of $j$.
This can be done if $n_1\geq 6$ of course, and also can be done if $n_1<6$ since then as $n\geq 6$,
there are two deleting cycles and we can choose $i, j$ in different cycles.

If $i'j'\not\in E$ then they are in the same deleting cycle and are adjacent in the cycle.
Hence $|G_1(i')\cap G_1(j')|\geq n-4$, where the excluding four vertices are  $i', j',$ the other neighbor $a$ of $i'$, the other neighbor $b$ of $j'$, and $a=b$ iff $i', j'$ are inside $C_3$. This proves the line of $\mu(G).$

 The only possible for a union of cycles to be strongly regular is when the cycles are all triangles. 
Hence $G$ is strongly iff $n_i=3$ for all $i$. The $\ell_1(G)=n$ and $\ell_{n-1}(G)=n-3$ are determined from (\ref{srg}) by using $\lambda=n-6$ and $\mu=n-3$. Therefore, we complete the proof.
\end{proof}

The Laplacian eigenvalues of a cycle is well-known \cite[Section 1.4.3]{spectra}, indeed \begin{equation}\label{cycle1}
 \ell_1(C_{s})=\left\{ \begin{array}{cl} 4, & s \text{ is even};  \\ 2+2\cos(\pi/s), & s \text{ is odd},   \end{array} \right.\end{equation} and \begin{equation}\label{cyclen-1}\ell_{s-1}(C_s)=2-\cos(2\pi/s).\end{equation}

\begin{prop}\label{n-3_2}
The $(n-3)$-regular graph $G=K_n-(C_{n_1}\cup C_{n_2}\cup \cdots \cup C_{n_t})$ with some $m_i>3$ has
$\alpha_1(G)=\alpha_2(G)=n,$ $\beta_1(G)=\beta_2(G)=n-4,$
$$\ell_1(G)=\left\{ \begin{array}{cl} n-\cos(2\pi/n), & t=1;  \\ n, & t\geq2   \end{array} \right.$$  and  $$\ell_{n-1}(G)=\left\{ \begin{array}{cl} n-2-2\cos(\pi/n_1), & n_i \text{ is odd for all }i;   \\ n-4, & \text{otherwise.}   \end{array} \right.$$
\end{prop}
\begin{proof} Applying $\lambda(G)=n-6$ and $\mu(G)=n-4$ to the definitions, one finds  
$\alpha_1(G)=\alpha_2(G)=n$ and $\beta_1(G)=\beta_2(G)=n-4$ immediately. From (\ref{cyclen-1}), 
$$\ell_1(G)=n-\ell_{n-1}(C_{n_1}\cup C_{n_2}\cup \cdots \cup C_{n_t})=\left\{
                                                                        \begin{array}{ll}
                                                                          n-\cos(2\pi/n), & \hbox{$t=1$;} \\
                                                                          n, & \hbox{$t\geq 2$.}
                                                                        \end{array}
                                                                      \right.$$ 
From (\ref{cycle1}), 
$$\ell_{n-1}(G)=n-\ell_{1}(C_{n_1}\cup C_{n_2}\cup \cdots \cup C_{n_t})
=\left\{
                                                                          \begin{array}{ll}
                                                                            n-2-2\cos(\pi/n_1), & \hbox{if $n_i$ is odd for all $i$;} \\
                                                                            n-4, & \hbox{otherwise.}
                                                                          \end{array}
                                                                        \right.$$
\end{proof}

From Proposition \ref{n-3_2}, we find that if $t\geq2$ and $n_i$ is even for some $i$ then the regular $(n-3)$-regular graph $G=K_n-(C_{n_1}\cup C_{n_2}\cup \cdots \cup C_{n_t})$ is extremal for (\ref{maineq}) with $\ell=\ell_1(G),\ell_{n-1}(G)$, the three inequalities in Corollary~\ref{eig}, (\ref{ee3}), (\ref{ee5}) and (\ref{ee7}).

\section*{Acknowledgments}

This research is supported by the
Ministry of Science and Technology of Taiwan R.O.C.
under the project NSC 102-2115-M-009-009-MY3.

\end{document}